\documentclass[12pt]{article}
\usepackage{amssymb}
\usepackage{amsmath}
\usepackage{amsthm}
\usepackage{graphicx}
\newtheorem{theorem}{Theorem}
\newtheorem{proposition}{Proposition}[section]
\newtheorem{definition}{Definition}[section]
\newtheorem{lemma}{Lemma}[section]

\newtheorem{remark}{Remark}[section]

\newtheorem{example}{Example}[section]

\renewenvironment{proof}{
\noindent{\bf Proof.}\rm} {\mbox{}\hfill\rule{0.5em}{0.809em}\par}

\begin{document}

\title{GR\"{O}BNER-SHIRSHOV BASES METHOD FOR  GELFAND-DORFMAN-NOVIKOV ALGEBRAS\footnote{Supported by the NNSF of China (11171118, 11571121) and
the Program on International Cooperation and Innovation,   Department
of Education,   Guangdong Province (2012gjhz0007).}}

\author{
L. A. Bokut\footnote {Supported by Russian Science Foundation (project
14-21-00065).} \\
{\small \ School of Mathematical Sciences, South China Normal
University}\\
{\small Guangzhou 510631, P. R. China}\\
{\small Sobolev Institute of Mathematics}\\
{\small pr. Akademika Koptyuga 4, Novosibirsk, Russian Federation}\\
{\small Department of Mathematics and Mechanics, Novosibirsk State University}\\
{\small  ul. Pirogova 2, Novosibirsk, Russian Federation}\\
{\small  bokut@math.nsc.ru}\\
\\
 Yuqun
Chen\footnote {Corresponding author.} \  and Zerui Zhang\\
{\small \ School of Mathematical Sciences, South China Normal
University}\\
{\small Guangzhou 510631, P. R. China}\\
{\small yqchen@scnu.edu.cn}\\
{\small 295841340@qq.com}}

\date{}

\maketitle

\noindent\textbf{Abstract:} We establish
Gr\"{o}bner-Shirshov bases theory for Gelfand-Dorfman-Novikov algebras over a field of characteristic $0$. As applications,
a PBW type theorem in Shirshov form is given and we provide an  algorithm for solving  the word problem of  Gelfand-Dorfman-Novikov algebras with finite homogeneous relations. We also construct a subalgebra of one generated free Gelfand-Dorfman-Novikov algebra which is not free.




\section{Introduction}\label{Intro}

Gelfand-Dorfman-Novikov algebras were introduced by I.M. Gelfand,
I.Ya. Dorfman \cite{Gelfand},  1979 in connection with Hamiltonian operators in the formal calculus of variations and A.A. Balinskii,   S.P. Novikov  \cite{Novikov 1},  1985 in connection with linear Poisson brackets of hydrodynamic type.  As it was pointed out in \cite{Novikov 2},  1985,   E.I. Zelmanov answered to a Novikov's question about simple finite dimensional Gelfand-Dorfman-Novikov algebras over a field of characteristic zero at the same year. He proved that there are no such non-trivial algebras,   see  \cite{Zelmanov},  1987. In 1989,   V.T. Filippov found first examples of simple infinite dimensional Gelfand-Dorfman-Novikov algebras  of characteristic $p\geq 0$ and simple finite dimensional Gelfand-Dorfman-Novikov algebras of characteristic $p>0$,   see \cite{Filippov 1}.  J.M. Osborn \cite{Osborn1,Osborn2,Osborn3},  1992-1994 gave the name Novikov algebra (he knew both papers \cite{Novikov 1,Gelfand}) and  began to classify simple finite dimensional Gelfand-Dorfman-Novikov algebras with prime characteristic $p>0$  and infinite dimensional ones with characteristic 0,   as well as irreducible modules \cite{Osborn4,Osborn5},  1995. Considering the contribution of Gelfand and Dorfman to Novikov algebras, we call Novikov algebras  as Gelfand-Dorfman-Novikov algebras in this paper.  There are also quite a few papers on the structure theory
(see,   for example,   X. Xu \cite{Xiaoping Xu 1,   Xiaoping Xu 2,   Xiaoping Xu 3,   Xiaoping Xu 4},   1995-2000,   C. Bai and D. Meng  \cite{Bai 1,   Bai 2,   Bai 3},   2001,   L. Chen,   Y. Niu and D. Meng \cite {Chen ly},   2008,   D. Burde and K. Dekimpe \cite{Burde 3},   2006)
and combinatorial theory of Gelfand-Dorfman-Novikov algebras,   and irreducible modules over Gelfand-Dorfman-Novikov algebras,   with applications to mathematics and mathematical physics.
The present paper is on combinatorial method of Gr\"{o}bner-Shirshov bases for Gelfand-Dorfman-Novikov algebras. Let us observe some combinatorial results. In \cite{Gelfand},   it was given an important observation by
 S.I. Gelfand that any differential commutative associative algebra  is a Gelfand-Dorfman-Novikov algebra under the new product $a\circ b=(Da)b$. This observation leads to a notion of the universal enveloping of a Gelfand-Dorfman-Novikov algebra (see and cf. \cite{trees}) that we use in the present paper. V.T. Filippov \cite{Filippov 2},  2001 proved that any Novikov nil-algebra of  nil-index $n$ with characteristic 0 is nilpotent (an analogy of Nagata-Higman theorem). He used essentially the Zelmanov theorem \cite{Zelmanov 2},  1988 that any Engel Lie algebra of index $n$ in characteristic 0 is nilpotent.
By the way,   Zelmanov \cite{Zelmanov 3},  1989 also proved the local nilpotency of any Engel Lie algebra of index $n$ in any characteristic. In 2002,   A. Dzhumadil'daev and  C. L\"{o}fwall \cite{trees} found structure of a free Gelfand-Dorfman-Novikov algebra
using trees and a free differential commutative algebra. We use essentially this result here.
Dzhumadil'daev \cite{codimension},   2011 found another linear basis of a free Gelfand-Dorfman-Novikov algebra using Young tableaux. This result was essentially used by L. Makar-Limanov and U. Umirbaev \cite {Freedom} in a proof of the Freiheitssatz theorem for Gelfand-Dorfman-Novikov algebras. Also they proved that the basic rank of the variety of Gelfand-Dorfman-Novikov algebras is one.

In the present paper,   we introduce Gr\"{o}bner-Shirshov bases method for Gelfand-Dorfman-Novikov algebras,   prove a
PBW type theorem in Shirshov form for Gelfand-Dorfman-Novikov algebras,   and provide an algorithm for solving the word problem for Gelfand-Dorfman-Novikov algebras with finite number of homogeneous relations.

Gr\"{o}bner and Gr\"{o}bner-Shirshov bases methods were invented by A.I. Shirshov,   a student of A.G. Kurosh,   for Lie algebras
and implicitely for associative algebras \cite{Sh62b},   1962 and non-associative (commutative anti-commutative) algebras \cite{Shirshov commutative},  1954,   by H. Hironaka for commutative topological algebras \cite{Hi64},   1964
and by B. Buchberger \cite{Bu70},   1965 for commutative algebras. As a prehistory, see A.I. Zhukov \cite{Zhukov}, 1950, another Kurosh's student.

Gr\"obner-Shirshov (Gr\"obner) bases methods deal with varieties and categories of (differential,   integro-differential,   PBW,   Leavitt,   Temperley-Lieb,   Iwahori-Hecke,   quadratic,   free products of two, over a commutative algebra, \dots) associative algebras,   (plactic,   Chines,   inverse,  \dots) semigroup algebras,   (Coxeter,   braid,   Artin-Tits,   Novikov-Boone,  \dots) group algebras,   semiring algebras,   Lie (restricted,   super-,   semisimple,   Kac-Moody,   quantum,   Drinfeld-Kohno,   over a commutative algebra, metabelian,   \dots) algebras,   associative conformal algebras,   Loday's (Leibniz,   di-,   dendriform) algebras,   Rota-Baxter algebras,   pre-Lie (i.e.,   right symmetric) algebras,   (simplicial,   strict monoidal,  \dots) categories,   non-associative (commutative,   anti-commutative,   Akivis,   Sabinin,  \dots) algebras,   symmetric (non-symmetric) operads,   $\Omega$-algebras,    modules,   and so on. Gr\"{o}bner-Shirshov bases method
is useful in homological algebra (Anick resolutions,   (Hochschild) cohomology rings of (Leavitt,   plactic,  \dots)
algebras),   in proofs of PBW type theorems (Lie algebra -- associative algebra,   Lie algebra --  pre-Lie algebra,   Leibniz algebra -- associative dialgebra,   Akivis algebra -- non-associative algebra,   Sabinin algebra - modules),   in algorithmic problems of algebras (solvable and unsolvable algorithmic problems),   in the theory of automatic groups and semigroups,   in independent constructions of Hall,
Hall-Shirshov  and Lyndon-Shirshov bases of a free Lie algebra,    on embedding theorems and many other applications.
For details one may see,   for example,   new surveys \cite{Bo2014}  and \cite{PBW}.

In this paper we prove a PBW type theorem in Shirshov form for Gelfand-Dorfman-Novikov algebras. The first of this kind of theorems is the following (see \cite{BokutKLM,  BoMa97,  BoMa99}).

Let $L=Lie(X|S)$ be a Lie algebra,   presented by generators $X$ and defining relations $S$ over a field $ k$ ,   $U(L)=k \langle X|S^{(-)}\rangle$  the universal enveloping associative algebra of  $L$ (here $S \Rightarrow  S^{(-)}$ using $[x,  y]\Rightarrow xy-yx$). Then  $S$ is a Lie Gr\"{o}bner-Shirshov basis in $Lie(X)$ if and only if $S^{(-)}$ is an associative Gr\"{o}bner-Shirshov basis in $k\langle X\rangle$.

As a corollary, let $S$ be a Lie Gr\"{o}bner-Shirshov basis (in particular,   $S$ be a multiplication table of $L$).  Then

(i) A linear basis of $U(L)$ consists of words $u_1u_2\dots u_k$,   $k\geq 0$,   where $u_{i}$'s are $S^{(-)}$-irreducible associative
Lyndon-Shirshov words (without brackets) in $X$,   $u_1\leq u_2\leq \dots \leq u_k$ in lexicographical order (meaning $a>ab$,   if $b\neq 1$) (in particular, a linear basis of $U(L)$ is PBW one if $S$ is a miltiplication table of $L$).

(ii) A linear basis of $U(L)$ consists of words $[u_1][u_2]\dots [u_k]$,   $k\geq 0$,   where $[u_i]$'s are $S$-irreducible
Lyndon-Shirshov Lie words in $X$,   $u_1\leq u_2\leq ...\leq u_k$ in lexicographical order.

(iii) A linear basis of $L$ consists of words $[u]$,   where $[u]$'s are $S$-irreducible
Lyndon-Shirshov Lie words in $X$.

For Gelfand-Dorfman-Novikov algebras we prove the following PBW type theorem in Shirshov form.

Let $GDN(X)$ be a free Gelfand-Dorfman-Novikov algebra, $k\{X\}$ be a free commutative differential algebra, $S\subseteq GDN(X)$ and $S^{c}$ a Gr\"{o}bner-Shirshov basis in $k\{X\}$,   which is obtained from $S$ by Buchberger-Shirshov algorithm in $k\{X\}$. Then

(i) $S'=\{uD^{m}s\mid s\in S^{c},   u\in [D^{\omega}X],   wt(u\overline{D^{m}s})=-1,   m\in \mathbb{N}\}$ is a Gr\"{o}bner-Shirshov basis in $GDN(X)$.

(ii)  The set $Irr(S')=\{w\in [D^{\omega}X] \mid w\neq u\overline{D^{t}s},    u\in [D^{\omega}X],   t\in \mathbb{N},   s\in S^{c},  wt(w)=-1\}=GDN(X)\cap Irr[S^{c}]$ is
a linear basis of $GDN(X|S)$. Thus,   any Gelfand-Dorfman-Novikov algebra $GDN(X|S)$ is embeddable into its universal enveloping commutative differential algebra $k\{X|S\}$.

Using Buchberger-Shirshov algorithm, we provide    algorithms for solving both the word problem for commutative differential algebras with finite number of $D\cup X$-homogeneous
defining relations and the word problem for Gelfand-Dorfman-Novikov algebras with finite number of $X$-homogeneous defining relations.
For Lie algebras it was proved by Shirshov in his original paper \cite{Sh62b},   see also \cite{Shir3}.  In general,   word problem for Lie algebras is unsolvable, see  \cite{Bo72}.  For Gelfand-Dorfman-Novikov algebras it remains unknown.
So far, the word problem (membership problem) for commutative differential algebras is solved mainly for the following cases \cite{word problem}:  radical ideals, isobaric (i.e., homogeneous with respect to  derivations) ideals, ideals with a finite or parametrical standard basis, and ideals generated by a composition of two differential polynomials (under some additional assumptions).

Finally,   we prove that the variety of Gelfand-Dorfman-Novikov algebras is not a Schreier one, i.e., not each subalgebra of a free Gelfand-Dorfman-Novikov algebra is free. The most famous  Schreier variety are the variety of groups \cite{Schreier}, the variety of non-associative  algebras \cite{Kurosh}, the variety of (non-associative) commutative and anti-commutative algebras \cite{Shirshov commutative}, the variety of Lie algebras \cite{Shirshov lie, witt}. For more details, see \cite{variety property, Umirbaev}.

\section{Free Gelfand-Dorfman-Novikov algebras}

A non-associative algebra
$A=(A,  \circ)$ is called a right-Gelfand-Dorfman-Novikov algebra \cite{codimension},   if $A$ satisfies the identities
$$
x\circ(y\circ z)-(x\circ y)\circ z=x\circ(z\circ y)-(x\circ z)\circ y,
$$
$$
x\circ (y\circ z)=y\circ(x\circ z).
$$

In the papers \cite{codimension,  trees},   the authors constructed the free Gelfand-Dorfman-Novikov algebra $GDN(X)$ generated by $X$ as follows:
A Young diagram is a set of boxes with
non-increasing numbers of boxes in each row. Rows and columns are numbered from top to bottom and from left to right.
Let $p$ be the number of rows and $r_i$ be the number of boxes in the $i$th row. To construct Gelfand-Dorfman-Novikov diagram,
we need to complement Young diagram by
one box in the first row. To construct Gelfand-Dorfman-Novikov tableaux on a well-ordered set $X$,   we need to fill Gelfand-Dorfman-Novikov diagrams
 by elements of $X$. Denote by $a_{i,  j}$ an element of $X$ in the box that is the cross of the $i$th row by the $j$th column. The filling rule is the following:

(a) $a_{i,  1}\geq a_{i+1,  1}$,     if $r_{i}=r_{i+1},  $ $  i = 1,   2,  \dots  ,    p-1$;

(b) The sequence $a_{p,  2}\cdots  a_{p,  r_{p}}a_{p-1,  2}\cdots  a_{p-1,  r_{p-1}}\cdots  a_{1,  2}\cdots  a_{1,  r_{1}+1}$   is non-decreasing.\\
Such a Gelfand-Dorfman-Novikov tableau corresponds to the following element of the free Gelfand-Dorfman-Novikov algebra:
\begin{eqnarray*}
 w&=&Y_{p}\circ(Y_{p-1}\circ(\cdots \circ(Y_{2}\circ Y_{1})\cdots  ))\ \ \mbox{(right-normed bracketing),  where }\\
 Y_{i}&=&(\cdots ((a_{i,  1}\circ a_{i,  2})\circ a_{i,  3}) \cdots \circ a_{i-1,  r_{i}-1})\circ a_{i,  r_{i}},   \ \  2\leq i \leq p,   \\
Y_{1}&=&(\cdots ((a_{1,  1}\circ a_{1,  2})\circ a_{1,  3}) \cdots \circ a_{1,  r_{1}})\circ a_{1,  r_{1}+1}
\end{eqnarray*}
(each $Y_{j}$ left-normed bracketing). In this case,   we say $w$ has degree $r_p+r_{p-1}+\dots +r_1+1$. We call such a $w$ as a Gelfand-Dorfman-Novikov tableau as well.
Such elements form  a linear basis of a free Gelfand-Dorfman-Novikov algebra generated by $X$ and we denote such  free Gelfand-Dorfman-Novikov algebra as $GDN(X)$, see  \cite{codimension,  trees}.

A commutative differential algebra $A=(A, \cdot, D)$  is a commutative associative algebra with one linear operator $D:\ A\rightarrow A$ such that
for any $a,b\in A,\ D(ab)=(Da)b+a(Db)$. We call such a $D$ a derivation of $A$.

Given a well-ordered set $X=\{a,   \ b,   \ c,   \dots  \}$,   denote $D^{\omega}X=\{D^{i}a\mid i\in \mathbb{N},  \ a\in X \}$,   $[D^{\omega}X]$ the free
 commutative monoid generated by $D^{\omega}X$ and $k$ a field of characteristic $0$.
 Let $D(1)=0,   D^{0}a=a,   D(D^{i}a)=D^{i+1}a$,   $D(\alpha u+\beta v)=\alpha Du+\beta Dv$ and $D(uv)=(Du)\cdot v+u\cdot D(v)$ for
  any $a\in X,  \ \alpha,   \beta \in k,  \ u,  v \in [D^{\omega}X]$ ($\cdot$ is often omitted). Then $(k[D^{\omega}X],   \ \cdot,   \ D)$ is a
  free commutative differential algebra over $k$,   see  \cite{differential algebra}.
  From now on we
denote $a$ as $a[-1]$,   $D^{i+1}a$ as $a[i]$,    and $(k[D^{\omega}X],   \ \cdot,   \ D)$ as $k[D^{\omega}X]$ or $k\{X\}$. Then $k\{X\}$ has a $k$-basis
as the set (also denote as $[D^{\omega}X]$) of all words of the form
 $$
 w=a_{n}[i_n]a_{n-1}[i_{n-1}]\cdots    a_{1}[i_1] \ or \ w=1,
 $$
 where $a_t\in X,  \ i_{t}\geq -1,    \ 1\leq t\leq n,   \ n\in \mathbb{N} $
and
$ (i_{n},   a_{n})\geq (i_{n-1},   a_{n-1}) \geq \dots  \geq (i_{1},  a_{1})\mbox{ lexicographically}$.
 For such $w\neq 1$,   we define the weight of $w$,   denoted by $wt(w)$,   to be $wt(w)=i_1+i_2+\dots +i_n$;
 the length of $w$,   denoted by $|w|$,   to be $|w|=n$; and the $D\cup X$-length of $w$,   denoted by $|w|_{D\cup X}$,
 to be $|w|_{D\cup X}=wt(w)+2n$,   which is exactly the number of $D$ and generators from $X$ that occur in $w$.
 For $w=1$,   define $wt(w)=|w|=|w|_{D\cup X}=0$. Furthermore,   if we define $\circ$ as
$$
f\circ g=(Df)g,  \ \ \  f,  g\in k\{X\},
$$
 then $(k\{X\},  \
 \circ)$ becomes a right-Gelfand-Dorfman-Novikov algebra. Its subspace
$$ span_{k}\{w\in [D^{\omega}X] \mid wt(w)=-1\},$$
is  a  subalgebra of $(k\{X\},  \
 \circ)$ (as Gelfand-Dorfman-Novikov algebra), we denote such subalgebra as $ GDN_{-1}(X)$. In  \cite{trees},  the authors showed that  the Gelfand-Dorfman-Novikov algebra homomorphism
$\varphi :GDN(X)\longrightarrow GDN_{-1}(X)$,   induced by $\varphi(a) =a[-1]$,   is an isomorphism. Therefore, $GDN_{-1}(X)$ is a free Gelfand-Dorfman-Novikov algebra generated by $X$,  which has a  $k$-basis $\{w\in [D^{\omega}X] \mid wt(w)=-1\}$. From now on, when no ambiguity arises,  we denote  both $GDN(X)$ and $GDN_{-1}(X)$ as $GDN(X)$ for convenient.

\section{Composition-Diamond lemmas}

\subsection{Monomial order}

We order $[D^{\omega}X]$ as follows.

For any $a[i],   b[j] \in D^{\omega}X$,   define
$$
a[i]< b[j] \ \Leftrightarrow \ (i,   a)<(j,  b)\  \mbox{ lexicographically}.
$$
For any $w=a_{n}[i_n]\cdots   a_{1}[i_1]\in [D^{\omega}X]$ with
 $a_{n}[i_n]\geq  \dots \geq a_{1}[i_1]$,   define
$$
ord(w)\triangleq (|w|,   a_{n}[i_n],  \dots  ,   a_{1}[i_1]).
$$

Then,   for any $u,  v \in [D^{\omega}X]$ we define
$$
u<v\Leftrightarrow ord(u)<ord(v) \ \mbox{ lexicographically}.
$$

It is clear that this is a well order on $[D^{\omega}X]$. We will use this order throughout this paper.

For any  $f\in k\{X\}$,    $\overline{f}$ means the leading word of $f$.  We denote the coefficient of $\overline{f}$ as $LC(f)$.

\begin{lemma}
Let the order $<$ on $[D^{\omega}X]$ be as above. Then
 \begin{enumerate}
 \item[(i)]\  $
 u<v\Rightarrow u\cdot w<v\cdot w,    \ \overline{Du}<\overline{Dv}
$
for any $u,  v,  w \in [D^{\omega}X],   \ u\neq 1 $.

\item[(ii)]\   $u<v\Rightarrow \overline{w\circ u}< \overline{w \circ v},  \ \overline{u\circ w}< \overline{v \circ w}
$
for any $u,  v,  w\in [D^{\omega}X]\setminus\{1\} $.

 \end{enumerate}

\end{lemma}\label{monomial}
\begin{proof}
(i) Noting that $\cdot$ is commutative and associative,   it is easy to see that $u<v\Rightarrow u\cdot w<v\cdot w$.
For any $w=a_{n}[i_n]\cdots   a_{1}[i_1]\neq 1$,   with
 $a_{n}[i_n]\geq \dots \geq a_{1}[i_1]$,   we have $ord(\overline{Dw})= (|w|,   a_{n}[i_n+1],   \dots  ,   a_{1}[i_1])$,   so $u<v\Rightarrow  \overline{Du}<\overline{Dv}$.

 (ii) For any $u,  v,  w\in [D^{\omega}X]\setminus\{1\}$,   we have
 $\overline{u\circ w}=\overline{(Du)w}=\overline{Du}\cdot w$ and $\overline{v\circ w}=\overline{(Dv)w}=\overline{Dv}\cdot w$,
 so $u<v\Rightarrow \overline{u\circ w}<\overline{v\circ w}$. By the same reasoning,   $u<v\Rightarrow \overline{w\circ u}<\overline{w\circ v}$.
 \end{proof}

\subsection{ $S$-words  }

For any $S\subseteq k\{X\}$,   we denote $Id[S]$ the ideal of $k\{X\}$ generated by $S$ and
$$
k\{X|S\}\triangleq k\{X\}/ Id[S]
$$
the commutative
differential algebra generated by $X$ with defining relations $S$.
Since $Id[S]$ is closed under $\cdot$ and the derivation  $D$,   we have
$$Id[S]=span_{k}\{u D^{t}s \mid u\in [D^{\omega}X],   t\in \mathbb{N},   s\in S\}.$$
For any $u\in [D^{\omega}X],   t\in \mathbb{N},   s\in S $,   we call $u D^{t}s$ an $S$-word in $k\{X\}$. We call $u D^{t}s$ an $S$-word in $GDN(X)$ if $wt(\overline{uD^{t}s})=-1$ and $S\subseteq GDN(X)$.

Suppose $S\subseteq GDN(X)$  and denote $Id(S)$ the ideal of $GDN(X)$ generated by $S$. Then we have the following lemma.

\begin{lemma}\label{Novikov ideal}
Suppose $S\subseteq GDN(X)$. Then
$$
Id(S)=span_{k}\{uD^{t}s \mid u\in [D^{\omega}X],   t\in \mathbb{N},    s\in S,    wt(\overline{uD^{t}s})=-1\}.
$$
\end{lemma}
\begin{proof}
It is clear that the right part is an ideal that contains $S$. We just need to show that $uD^{t}s \in Id(S)$
whenever $wt(\overline{uD^{t}s})=-1$. Since $wt(\overline{uD^{t}s})=-1$,
we have
$$ uD^{t}s=c_{1}[i_1]\cdots   c_{n}[i_n]a_{1}[-1]\cdots   a_{m}[-1](D^{t}s)b_{1}[-1]\cdots   b_{t}[-1],  $$
where
$u=c_{1}[i_1]\cdots   c_{n}[i_n]a_{1}[-1]\cdots   a_{m}[-1]b_{1}[-1]\cdots   b_{t}[-1]$,
$\ m=i_1+\dots +i_n$ and $ i_n\geq i_{n-1}\geq \dots \geq i_1\geq 0.$
So the lemma will be clear if we show
 \begin{enumerate}
 \item[$(i)$]\   $(D^{t}s)b_{1}[-1]\cdots   b_{t}[-1]\in Id(S)$ whenever $s\in S$;
 \item[$(ii)$]\   $c[p]a_{1}[-1]\cdots   a_{p}[-1]f\in Id(S)$ whenever $f\in Id(S)$.
 \end{enumerate}

 To prove $(i),  $  we use induction on $t$. If $t=0,  $ it is clear.
 Suppose that it holds for all $t\leq n$. Then
\begin{eqnarray*}
&&(D^{n+1}s)b_{1}[-1]\cdots   b_{n+1}[-1]\\
&=&((D^{n}s)b_{1}[-1]\cdots   b_{n}[-1])\circ b_{n+1}[-1]\\
&&-\sum_{1\leq i \leq n}(D^{n}s)b_{1}[-1]\cdots   (Db_{i}[-1])\cdots b_{n}[-1]\cdot b_{n+1}[-1]\\
&=&((D^{n}s)b_{1}[-1]\cdots   b_{n}[-1])\circ b_{n+1}[-1]\\
&&-\sum_{1\leq i \leq n}b_{i}[-1]\circ ((D^{n}s)b_{1}[-1]\cdots   b_{i-1}[-1]b_{i+1}[-1]\cdots    b_{n+1}[-1])\\
&\in & Id(S).
\end{eqnarray*}

To prove $(ii)$,   we use induction on $p$.
If $p=0,  $ it is clear.
 Suppose that it holds for all $p\leq n$. Then
\begin{eqnarray*}
&&c[n+1]a_{1}[-1]\cdots   a_{n+1}[-1]f\\
&=&(c[n]a_{1}[-1]\cdots   a_{n+1}[-1])\circ f\\
&&-\sum_{1\leq i \leq n+1}c[n]a_{1}[-1]\cdots   (Da_{i}[-1])\cdots a_{n+1}[-1]\cdot f\\
&=&(c[n]a_{1}[-1]\cdots   a_{n+1}[-1])\circ f\\
&&-\sum_{1\leq i \leq n+1}c[n]a_{1}[-1]\cdots   a_{i-1}[-1]a_{i+1}[-1]\cdots   a_{n+1}[-1]\cdot (a_{i}[-1]\circ f)\\
&\in & Id(S).
\end{eqnarray*}
So $Id(S)=span_{k}\{uD^{t}s \mid u\in [D^{\omega}X],   t\in \mathbb{N},    s\in S,    wt(\overline{uD^{t}s})=-1\}$.
\end{proof}

Let  $S$ be a subset of $k\{X\}$. We call
$S$ homogeneous (weight homogeneous, $D\cup X$-homogeneous,   resp.),  if for any $f=\sum^{q}_{j=1}\beta_jw_{j}\in S$,   we have
 $|w_{1}|=\dots =|w_{q}| \ (wt(w_{1})=\dots =wt(w_{q}), |w_{1}|_{D\cup X}=\dots =|w_{q}|_{D\cup X},   \mbox{ resp.}) $.
We have the following lemma immediately.

\begin{lemma}\label{homogeneous}
Let $S\subseteq k\{X\}$,   $f=\sum_{i\in I} \beta_{i}u_{i}D^{t_{i}}s_{i}$,   where  each $\beta_{i}\in k,   u_{i}\in [D^{\omega}X],  s_{i}\in S,   t_i\in \mathbb{N}$.
If $f$ and $S$ are homogeneous (weight homogeneous, $D\cup X$-homogeneous,   resp.),   then we can suppose that
$|\overline{u_{i}D^{t_{i}}s_{i}}|=|\overline{f}| \ (wt(\overline{u_{i}D^{t_{i}}s_{i}})=wt(\overline{f}), |\overline{u_{i}D^{t_{i}}s_{i}}|_{D\cup X}=|\overline{f}|_{D\cup X},   \mbox{resp.})$ for any $i\in I$.
\end{lemma}

\subsection{Composition-Diamond lemma for commutative differential algebras}


The idea of this subsection is essentially the same as the construction of standard differential Gr\"{o}bner bases
in  \cite{Ferro,Ollivier},   in which the authors deal with more general case with several derivations.

For any $u,  v\in [D^{\omega}X]$,   we always denote $lcm(u,  v)$ the least common multiple of $u,  v$ in $[D^{\omega}X]$. We call $lcm(u,  v)$ a non-trivial least common multiple of $u$ and $v$ if $|lcm(u,  v)|<|uv|$.

For any $f,  g\in S\subseteq k\{X\}$,    if $w =lcm(\overline{D^{t_1}f},  \overline{D^{t_2}g})$
is a non-trivial least common multiple of $\overline{D^{t_1}f}$ and $\overline{D^{t_2}g}$,   then we call
$$ [D^{t_1}f,  D^{t_2}g]_{w}= \frac{1}{\alpha_{1}}w|_{_{\overline{D^{t_{1}}f}\mapsto D^{t_{1}}f}}-\frac{1}{\alpha_{2}}w|_{_{\overline{D^{t_{2}}g}\mapsto D^{t_{2}}g}}$$
a composition for $D^{t_1}f \wedge D^{t_2}g$ corresponding to $w$,   where $\alpha_1=LC(D^{t_1}f),  \alpha_2=LC(D^{t_2}g) $.

For a polynomial $h\in k\{X\}$,   we say $h\equiv0 \ mod(S,  w)$ if $h=\sum\gamma_iu_{i}D^{t_i}s_i$,   where each
$\gamma_i\in k$,    $ u_{i}D^{t_i}s_i$ is an $S$-word and $\overline{u_{i}D^{t_i}s_i}<w$. Denote $h\equiv h' \ mod(S,  w)$ if $h- h'\equiv 0 \ mod(S,  w)$.
The composition $[D^{t_1}f,  D^{t_2}g]_{w}$ is trivial $mod(S,  w)$ if
$[D^{t_1}f,  D^{t_2}g]_{w}\equiv0\  mod(S,  w)$.

For $f,  g\in S,   w=lcm(\overline{D^{t_{1}}f},   \overline{D^{t_{2}}g})$,   if $w=\overline{f}$ or $w=\overline{g}$,   then the composition is called inclusion; Otherwise,   the composition is called intersection.

\begin{definition}
Let  $S$ be a
non-empty subset of $k\{X\}$. Then the set $S$ is called a Gr\"{o}bner-Shirshov basis in
$k\{X\}$ if all compositions of $S$ in $k\{X\}$ are trivial.
\end{definition}



\begin{theorem}\label{differential cd lemma} (Composition-Diamond lemma for commutative differential algebras) \cite{Ferro,   Ollivier}
Let $<$ be the monomial order on $k\{X\}$ as before and  $S$  a
non-empty subset of $k\{X\}$. Let $Id[S]$ be the ideal of $k\{X\}$
generated by $S$. Then the following statements are equivalent.
\begin{enumerate}
\item[(i)] $S$ is a Gr\"{o}bner-Shirshov basis in $k\{X\}$.
\item[(ii)] $0\neq h\in{Id[S]}\Rightarrow\overline{h}=u\overline{D^{t}s}$ for
some $s\in S,   u\in [D^{\omega}X],  t\in \mathbb{N}$.
\item[(iii)]$Irr[S]=\{w \in [D^{\omega}X] \mid w\neq u\overline{D^{t}s},   \ u\in [D^{\omega}X],   t\in \mathbb{N},   s\in S\} $ is a linear basis for $k\{X|S\}$.
\end{enumerate}
\end{theorem}

\ \

\noindent{\bf Buchberger-Shirshov algorithm:} If a subset  $S\subset
k\{X\}$ is not a Gr\"{o}bner-Shirshov basis then
one can add all non-trivial compositions of $S$ to
$S$. Continuing this process repeatedly,   we finally obtain a
Gr\"{o}bner-Shirshov basis $S^{c}$ that contains $S$. Such a process
is called Buchberger-Shirshov algorithm.

\subsection{Composition-Diamond lemma for Gelfand-Dorfman-Novikov algebras}

For any $u,  v,   w \in [D^{\omega}X]$,   we call $w$ a common multiple of $u$ and $v$ in $ GDN(X)$ if
$wt(w)=-1$ and $w$ is a common multiple of $u$ and $v$ in $[D^{\omega}X]$; $w$ is a non-trivial common multiple of $u$
and $v$ in $ GDN(X)$ if $w$ is a common multiple of $u$ and $v$ in $ GDN(X)$ such that $w\neq uvw'$ for any $w'\in [D^{\omega}X]$.

Let $f,  g\in GDN(X)$ and $w$ a non-trivial common multiple of $\overline{D^{t_{1}}f}$ and $\overline{D^{t_{2}}g}$ in $GDN(X)$. Then
a composition of $D^{t_{1}}f\wedge D^{t_{2}}g$ relative to $w$ is defined as
$$
(D^{t_1}f,  D^{t_2}g)_{w}= \frac{1}{\alpha_{1}}w|_{_{\overline{D^{t_{1}}f}\mapsto D^{t_{1}}f}}-\frac{1}{\alpha_{2}}w|_{_{\overline{D^{t_{2}}g}\mapsto D^{t_{2}}g}},
$$
where $\alpha_1=LC(D^{t_1}f)$ and $\alpha_2=LC(D^{t_2}g)$.

Suppose that $S\subseteq GDN(X)$ and $h\in GDN(X)$. Then we say $h\equiv0 \
 mod(S,  w)$ if $h=\sum\beta_iu_{i}D^{t_i}s_i$,   where each $\beta_i\in k,  \
u_{i}D^{t_i}s_i$ is an $S$-word such that $wt(\overline{u_{i}D^{t_i}s_i})=-1$ and $\overline{u_{i}D^{t_i}s_i}<w$.
The composition $(D^{t_1}f,  D^{t_2}g)_{w}$ is trivial $ mod(S,  w)$ if
$(D^{t_1}f,  D^{t_2}g)_{w}\equiv0\ \  mod(S,  w)$.

Let
$$
w=lcm(\overline{D^{t_1}f},  \overline{D^{t_2}g})d_{1}[m_1]\cdots   d_{p}[m_p]c_{1}[-1]\cdots   c_{q}[-1]
$$
be a non-trivial common multiple of $\overline{D^{t_{1}}f}$ and $\overline{D^{t_{2}}g}$ in $ GDN(X)$,   where $ d_1[m_1]\ge \dots \ge d_p[m_p], m_p> 0, c_1\ge \dots \ge c_q$. Then we say $w$
to be critical if one of the following holds:
\begin{enumerate}
\item[(i)] If $wt(lcm(\overline{D^{t_1}f},  \overline{D^{t_2}g})>-1$,   then $d_{1}[m_1]\cdots   d_{p}[m_p]$ is empty.
\item[(ii)] If $wt(lcm(\overline{D^{t_1}f},  \overline{D^{t_2}g})=-1$,   then $d_{1}[m_1]\cdots   d_{p}[m_p]c_{1}[-1]\cdots   c_{q}[-1]$ is empty.
\item[(iii)] If $wt(lcm(\overline{D^{t_1}f},  \overline{D^{t_2}g}))<-1$,   then $wt(lcm(\overline{D^{t_1}f},  \overline{D^{t_2}g})d_{1}[m_1]\cdots   d_{p-1}[m_{p-1}])<-1$
and $wt(lcm(\overline{D^{t_1}f},  \overline{D^{t_2}g})d_{1}[m_1]\cdots   d_{p}[m_p])\geq -1$.
\end{enumerate}

\begin{definition}
Let $S$ be a
non-empty subset of $GDN(X)$. Then the set $S$ is called a Gr\"{o}bner-Shirshov basis in
$GDN(X)$ if all compositions of $S$ in $GDN(X)$ are trivial.
\end{definition}

\begin{lemma}
Suppose that the composition
$(D^{t_1}f,  D^{t_2}g)_{w}$ is trivial for every critical common multiple $w$ of $\overline{D^{t_{1}}f}$ and $\overline{D^{t_2}g}$,
where $f,  g\in S$,   $t_1,  t_2\in \mathbb{N}$. Then $S$ is a Gr\"{o}bner-Shirshov basis in $GDN(X)$.
\end{lemma}
\begin{proof}
Noting that any common multiple of $\overline{D^{t_{1}}f}$ and $ \overline{D^{t_{2}}g}$ in $GDN(X)$ contains some critical common multiple $w$ of $\overline{D^{t_{1}}f}$ and $ \overline{D^{t_{2}}g}$,   the result follows.
\end{proof}

\begin{lemma}\label{novikov key lemma}
Suppose that $S$ is a Gr\"{o}bner-Shirshov basis in $GDN(X)$,   $f,  g\in S$ and $w=\overline{uD^{t}f}=\overline{vD^{t'}g}\in  [D^{\omega}X]$, $wt(w)=-1$.
Then $\frac{1}{\alpha_{1}}uD^{t}f-\frac{1}{\alpha_{2}}vD^{t'}g\equiv 0 \  mod(S,  w),  $ where $\alpha_1=LC(uD^{t}f)$ and $\alpha_2=LC(vD^{t'}g)$.
\end{lemma}
\begin{proof}
If $u=u'\overline{D^{t'}g}$ for some $u' \in [D^{\omega}X]$,   then $v=u'\overline{D^{t}f}$. Thus
\begin{eqnarray*}
&&\frac{1}{\alpha_{1}}uD^{t}f-\frac{1}{\alpha_{2}}vD^{t'}g\\
&=&\frac{1}{\alpha_{1}}u'(\overline{D^{t'}g})D^{t}f-\frac{1}{\alpha_{1}\alpha_{2}}u'(D^{t}f)D^{t'}g+
\frac{1}{\alpha_{1}\alpha_{2}}u'(D^{t}f)D^{t'}g-\frac{1}{\alpha_{2}}u'(\overline{D^{t}f})D^{t'}g\\
&=&\frac{1}{\alpha_{1}}(\overline{D^{t'}g}-\frac{1}{\alpha_{2}}D^{t'}g)u'D^{t}f-
\frac{1}{\alpha_{2}}(\overline{D^{t}f}-\frac{1}{\alpha_{1}}D^{t}f)u'D^{t'}g\\
&\equiv&0 \ mod(S,  w).
\end{eqnarray*}

Otherwise, $w$ is a non-trivial common multiple of $\overline{D^{t}f}$ and $\overline{D^{t'}g}$ in $GDN(X)$.
Since $S$ is a Gr\"{o}bner-Shirshov basis, by definition we have
$\frac{1}{\alpha_{1}}uD^{t}f-\frac{1}{\alpha_{2}}vD^{t'}g\equiv 0 \  mod(S,  w).  $
\end{proof}

\begin{lemma}\label{novikov two part}
Let $S$ be a non-empty subset of $GDN(X)$. Denote
$$Irr(S)=\{w\in [D^{\omega}X] \mid w\neq \overline{uD^{t}s},    u\in [D^{\omega}X],   t\in \mathbb{N},   s\in S,   wt(w)=-1 \}.$$
Then for all $h\in GDN(X)$,   we have
$$h=\sum_{u_{i}\overline{D^{t_{i}}s_{i}}\leq \overline{h}}\beta_{i}u_{i}D^{t_{i}}s_{i}+\sum_{w_{j}\leq \overline{h}}\gamma_{j}w_{j},  $$
where each $\beta_{i},  \gamma_{j}\in k,  \ u_{i}\in [D^{\omega}X],  \ t_{i}\in \mathbb{N},  \ s_{i}\in S,   \ w_{j}\in Irr(S),  $ and
$wt( u_{i}\overline{D^{t_{i}}s_{i}})=wt(w_{j})=-1.$
\end{lemma}
\begin{proof}
By induction on $\overline h$,   we have the result.
\end{proof}

\begin{theorem}\label{Novikov cd lemma} (Composition-Diamond lemma for Gelfand-Dorfman-Novikov algebras)
Let $S$ be  a
non-empty subset of $GDN(X)$ and $Id(S)$ be the ideal of $GDN(X)$
generated by $S$. Then the following statements are equivalent.
\begin{enumerate}
\item[(i)] $S$ is a Gr\"{o}bner-Shirshov basis in $GDN(X)$.
\item[(ii)] $0\neq h\in{Id(S)}\Rightarrow\overline{h}=u\overline{D^{t}s}$ for
some $s\in S,   u\in [D^{\omega}X],  t\in \mathbb{N}$.
\item[(iii)]$Irr(S)=\{w\in [D^{\omega}X] \mid w\neq u\overline{D^{t}s},    u\in [D^{\omega}X],   t\in \mathbb{N},   s\in S,  wt(w)=-1\} $ is a linear basis for $GDN(X|S)\triangleq GDN(X)/Id(S)$.
\end{enumerate}
\end{theorem}
\begin{proof}
$(i)\Rightarrow (ii)$. Let $S$ be a Gr\"{o}bner-Shirshov basis and
$0\neq h\in Id(S).$ Then $h$ has an
expression $h=\sum_{i=1}^n\beta_iu_iD^{t_i}s_{i}$,
where  each $0\neq\beta_i\in k,   \ u_i\in
[D^{\omega}X],  t_{i}\in \mathbb{N},    \ s_i\in S, wt(u_i\overline{D^{t_i}s_{i}})=-1$. Denote $w_i=u_i\overline{D^{t_i}s_{i}},  \ i=1,  2,  \dots ,   n.$ We may assume without
loss of generality that
$$w_1=w_2=\dots  =w_l> w_{l+1}\geq w_{l+2}\geq \dots $$
\noindent for some $l\geq 1$. Then $w_{1}\geq \overline{h}$.

We show the result by induction on $(w_1,  l)$,   where for any $l,  l'\in
\mathbb{N}$ and $w,  w'\in [D^{\omega}X]$,   $(w,  l)<(w',  l')$
lexicographically. We call $(w_1,  l)$ the height of $h$.

If $\overline{h}=w_{1}$ or $l=1$,   then the result is obvious.

Now suppose that $w_{1}>\overline{h}$. Then $l>1$ and $u_1\overline{D^{t_1}s_{1}}=u_2\overline{D^{t_2}s_{2}}$. By Lemma
\ref{novikov key lemma},   we have
\begin{eqnarray*}
&&\beta_1u_1D^{t_1}s_{1}+\beta_2u_2D^{t_2}s_{2}\\
&=&\beta_1(u_1D^{t_1}s_{1}-\frac{\alpha_{1}}{\alpha_{2}}u_2D^{t_2}s_{2})+\frac{\alpha_{1}\beta_{1}+\alpha_{2}\beta_{2}}{\alpha_{2}}u_2D^{t_2}s_{2}\\
&\equiv&\frac{\alpha_{1}\beta_{1}+\alpha_{2}\beta_{2}}{\alpha_{2}}u_2D^{t_2}s_{2} \ \ \ \
mod(S,  w_1),
\end{eqnarray*}
 where $\alpha_{i}=LC(u_iD^{t_i}s_{i})$,   $i=1,  2$. Thus,
\begin{eqnarray*}
h&=&\frac{\alpha_{1}\beta_{1}+\alpha_{2}\beta_{2}}{\alpha_{2}}u_2D^{t_2}s_{2}+\sum_{i=3}^n\beta_iu_iD^{t_i}s_{i}
+\sum_{j=1}^m\gamma_jv_jD^{t'_j}s'_{j},   \ \  v_j\overline{D^{t'_j}s'_{j}}<w_{1},
\end{eqnarray*}
which has height $<(w_1,  l)$. Now the result follows by induction.

$(ii)\Rightarrow (iii).$ By Lemma \ref{novikov two part},   the set
$Irr(S)$ generates the algebra $GDN(X|S)$ as a
$k$-vector space. On the other hand,   suppose that $\sum_{1\leq i\leq n}\gamma_iw_i=0$ in
$GDN(X|S)$,   where each $0\neq \gamma_i\in k$,   $w_i\in Irr(S)$ and $w_1>w_2>\dots >w_n$.
 Then we have $\sum_{1\leq i\leq n}\gamma_iw_i=\sum_{1\leq j\leq m}\beta_ju_jD^{t_j}s_{j}\neq 0$ in $GDN(X)$. So by $(ii)$ we get $w_{1}\notin Irr(S)$,    which
contradicts to the choice of $w_{1}$.

$(iii)\Rightarrow (i).$
 For any $f,  g\in S,  \ t_1,  t_2\in \mathbb{N}$,   denote $w$ a non-trivial common multiple of $\overline{D^{t_1}f}$ and $  \overline{D^{t_2}g}$. Then by Lemma \ref{novikov two part}, we
have
$$(D^{t_1}f,  D^{t_2}g)_{w}=\sum_{u_{i}\overline{D^{t_{i}}s_{i}}< w}\beta_{i}u_{i}D^{t_{i}}s_{i}+\sum_{w_{i}<w}\gamma_{j}w_{j},  $$
where each $\beta_{i},  \gamma_{j}\in k,  \ u_{i}\in [D^{\omega}X],  \ t_{i}\in \mathbb{N},  \ w_{j}\in Irr(S)$ and
$wt( u_{i}\overline{D^{t_{i}}s_{i}})=wt(w_{j})=-1.$
Since $(D^{t_1}f,  D^{t_2}g)_{w}\in Id(S)$ and by $(iii)$,   we have
$$
(D^{t_1}f,  D^{t_2}g)_{w}\equiv 0 \mod(S,w).
$$
Therefore,   $S$ is a Gr\"{o}bner-Shirshov basis in $GDN(X)$.
\end{proof}
\ \

Since for any Gelfand-Dorfman-Novikov tableau
$$w=Y_{p}\circ(Y_{p-1}\circ(\cdots \circ(Y_{2}\circ Y_{1})\cdots  )),  $$
we have
$$ \overline{   w}=\overline{DY_{p}}\cdot \overline{DY_{p-1}} \cdots  \overline{D Y_{2}}\cdot \overline{Y_{1}}  $$
and
\begin{eqnarray*}
 \overline{D Y_{i}}&=&a_{i,  1}[r_{i}-1]a_{i,  r_{i}}[-1]\cdots   a_{i,  3}[-1]a_{i,  2}[-1],   \ \  2\leq i \leq p,   \\
\overline{Y_{1}}&=&a_{1,  1}[r_{1}-1]a_{1,  r_{1}+1}[-1]a_{1,  r_{1}}[-1]\cdots   a_{1,  3}[-1]a_{1,  2}[-1],
\end{eqnarray*}
where
\begin{eqnarray*}
 Y_{i}&=&(\cdots ((a_{i,  1}\circ a_{i,  2})\circ a_{i,  3}) \cdots \circ a_{i,  r_{i}},   \ \  2\leq i \leq p,   \\
Y_{1}&=&(\cdots ((a_{1,  1}\circ a_{1,  2})\circ a_{1,  3}) \cdots \circ a_{1,  r_{1}})\circ a_{1,  r_{1}+1},
\end{eqnarray*}
we immediately get the following proposition:
\begin{proposition} If $S\subseteq GDN(X)$ is a Gr\"{o}bner-Shirshov basis, then the set
$\{w \in GDN(X)\mid w $ is a Gelfand-Dorfman-Novikov  tableau,   $\overline{w}\in Irr(S) \}$ is a linear basis for $GDN(X|S)$.
\end{proposition}

\section{Applications}
\subsection{An example}
In the paper \cite{ example},   the authors list a lot of left-Gelfand-Dorfman-Novikov algebras in low dimensions.   We can get their corresponding right-Gelfand-Dorfman-Novikov algebras using  $a\circ_{op} b\triangleq b\circ a$,
see also \cite{codimension}.

\begin{example} (\cite{ example})
Let $X=\{e_{1},  e_{2},  e_{3},  e_{4}\},  \ S=\{ e_{2}[0] e_{1}[-1]=e_{3}[-1],   e_{3}[0] e_{1}[-1]=e_{4}[-1],  $ $e_{i}[0] e_{j}[-1]=0,  $
if $\ (i,  j)\notin \{(2,  1),  (3,  1)\},   1\leq i,  j\leq 4\}$. Then $S$ is a Gr\"{o}bner-Shirshov basis in $GDN(X)$. It follows from Theorem \ref{Novikov cd lemma} that $\{e_{1}[-1],  e_{2}[-1],  e_{3}[-1],  e_{4}[-1]\}$ is a linear basis of the Gelfand-Dorfman-Novikov algebra $GDN(X|S)$.
\end{example}
\begin{proof}
Denote
$$
f_{ij}: \ \ \ e_{i}[0]e_{j}[-1]=\sum_{1\leq l\leq 4}\alpha_{ij}^{l}e_{l}[-1] \in S,   \ 1\leq i,  j\leq 4.
$$
Before checking the compositions,   we prove the following claims.

Claim (i): Let $w=e_{i}[n]e_{i_{1}}[-1]\cdots $$e_{i_{n+1}}[-1]$,   $n\geq 0$.
Then $w=\sum\alpha_{j}u_{j}D^{t_{j}}s_{j},   $  where  each  $\overline{u_{j}D^{t_{j}}s_{j}}\leq w$,   $s_{j}\in S$ if  $ i_{l}\neq  1$ for some $1\leq l\leq n+1$.

We show Claim (i) by induction on $n$.
For $n=0$ or $1$,   the result follows immediately. Suppose $t\geq 2$ and the result holds for any $ n<t.$ Then
\begin{eqnarray*}
w&=&e_{i}[t]e_{i_{1}}[-1]\cdots   e_{i_{t+1}}[-1]\\
&=&D^{t}(e_{i}[0]e_{i_{1}}[-1]-\sum_{1\leq m\leq 4}\alpha_{i,  i_1}^{m}e_{m}[-1])e_{i_{2}}[-1]\cdots   e_{i_{t+1}}[-1]          \\
&&-\sum_{0\leq p\leq t-1}\binom{t}{p}e_{i}[p]e_{i_1}[t-1-p]e_{i_{2}}[-1]\cdots   e_{i_{t+1}}[-1]\\
&&+\sum_{1\leq m\leq 4}\alpha_{i,  i_1}^{m}e_{m}[t-1]e_{i_{2}}[-1]\cdots   e_{i_{t+1}}[-1].
\end{eqnarray*}
If  for any $1\leq l\leq n+1$,   $i_{l}\neq 1$,   then by induction hypothesis,   the result follows immediately.
Otherwise,   say $i_{1}= 1$,   then $ i_{l}\neq  1$ for some $2\leq l\leq n+1$.  By induction hypothesis,   the result follows immediately.

Claim (ii): For any $n_{1},   n_{2}\geq 0$,   we have $w=e_{l}[n_{1}]e_{i}[n_{2}]e_{i_{1}}[-1]\cdots   e_{i_{n_{1}+n_{2}+1}}[-1]=\sum\alpha_{j}u_{j}D^{t_{j}}s_{j},   \ \mbox{with \ each } \overline{u_{j}D^{t_{j}}s_{j}}\leq w.$

We show Claim (ii) by induction on $n_{1}$. If $n_{1}=0$,   then
\begin{eqnarray*}
w&=&(e_{l}[0]e_{i_{1}}[-1]-\sum_{1\leq m\leq 4}\alpha_{l,  i_1}^{m}e_{m}[-1])e_{i}[n_{2}]e_{i_{2}}[-1]\cdots   e_{i_{n_{1}+n_{2}+1}}[-1]\\
&&+\sum_{1\leq m\leq 4}\alpha_{l,  i_1}^{m}e_{m}[-1]e_{i}[n_{2}]e_{i_{2}}[-1]\cdots   e_{i_{n_{1}+n_{2}+1}}[-1].
\end{eqnarray*}
By Claim (i),   the result follows immediately.
If $n_{1}>0$,   then
\begin{eqnarray*}
w&=&D^{n_1}(e_{l}[0]e_{i_{1}}[-1]-\sum_{1\leq m\leq 4}\alpha_{l,  i_1}^{m}e_{m}[-1])e_{i}[n_{2}]e_{i_{2}}[-1]\cdots   e_{i_{n_{1}+n_{2}+1}}[-1] \\
&&-\sum_{0\leq p\leq n_{1}-1}\binom{n_1}{p}e_{l}[p]e_{i_1}[n_{1}-1-p]e_{i}[n_{2}]e_{i_{2}}[-1]\cdots   e_{i_{n_{1}+n_{2}+1}}[-1] \\
&&+\sum_{1\leq m\leq 4}\alpha_{l,  i_1}^{m}e_{m}[n_{1}-1]e_{i}[n_{2}]e_{i_{2}}[-1]\cdots   e_{i_{n_{1}+n_{2}+1}}[-1].
\end{eqnarray*}
By induction hypothesis,   the result follows immediately.

For any $ t\in \mathbb{N},   u\in [D^{\omega}X]$, if $wt(\overline{(D^{t}f_{ij})u})=-1$, $|u|>0$ and $\overline{(D^{t}f_{ij})u}\neq e_{i}[t](e_{1}[-1])^{t+1}$, then by Claims (i) and (ii),   we have
$$
(D^{t}f_{ij})u-\overline{(D^{t}f_{ij})u}=\sum_{0\leq p\leq t-1}e_{i}[p]e_{j}[t-1-p]u+\sum_{1\leq m\leq 4}\alpha_{i,  j}^{m}e_{m}[t-1]u\equiv 0 \ mod(S,\overline{(D^{t}f_{ij})u} ).
$$
Since for any $t\in \mathbb{N}$, $ j\neq l$, each critical common multiple of $D^{t}f_{ij}\wedge D^{t}f_{il}$ has form $w=e_{i}[t]e_{j}[-1]e_{l}[-1]e_{i_{1}}[-1]\cdots e_{i_{t-1}}[-1]$, we get
$$
(D^{t}f_{ij},  D^{t}f_{il})_{w}\equiv w-w\equiv 0 \ mod(S,w).
$$
For the case of $D^{t_1}f_{i_{1}j}\wedge D^{t_2}f_{i_{2}j}$, where $t_{1}\neq t_{2} $ or $i_{1}\neq i_{2}$, the proof is almost the same.
So $S$ is a Gr\"{o}bner-Shirshov basis in $GDN(X)$.
\end{proof}

\subsection{PBW type theorem in Shirshov form}\label{subsection PBW}

\begin{theorem}(PBW type theorem in Shirshov form)\label{Novikov PBW}
Let $GDN(X)$ be a free Gelfand-Dorfman-Novikov algebra, $k\{X\}$ be a free commutative differential algebra, $S\subseteq GDN(X)$ and $S^{c}$ a Gr\"{o}bner-Shirshov basis in $k\{X\}$,   which is obtained from $S$ by Buchberger-Shirshov algorithm. Then

(i) $S'=\{uD^{m}s\mid s\in S^{c},   u\in [D^{\omega}X],   m\in \mathbb{N},   wt(u\overline{D^{m}s})=-1\}$ is a Gr\"{o}bner-Shirshov basis in $GDN(X)$.

(ii)  The set $Irr(S')=\{w\in [D^{\omega}X] \mid w\neq u\overline{D^{t}s},    u\in [D^{\omega}X],   t\in \mathbb{N},   s\in S^{c},  wt(w)=-1\}=GDN(X)\cap Irr[S^{c}]$ is
a linear basis of $GDN(X|S)$. Thus,   any Gelfand-Dorfman-Novikov algebra $GDN(X|S)$ is embeddable into its universal enveloping commutative differential algebra $k\{X|S\}$.
\end{theorem}
\begin{proof}
(i). We first show that any $h\in S^c$ has the form $h=\sum_{i\in I_{h}}\gamma_{i}w_{i},   \ \mbox{with each }  \gamma_{i}\neq0,   \ wt(w_{i})=wt(w_{i'}),    \ i,  \ i'\in I_{h}$. Suppose
$$
f=\sum_{i\in I_{f}}\beta_{i}w_{i},   \ \mbox{ with } wt(w_{i})=wt(w_{i'}) \mbox{ for any }  i,  \ i'\in I_{f},
$$
$$
g=\sum_{i\in I_{g}}\beta_{i}w_{i},   \ \mbox{ with } wt(w_{i})=wt(w_{i'}),   \mbox{ for any } i,  \ i'\in I_{g},
$$
and
$$
(D^{t}f,  D^{t'}g)_{w'}=\frac{1}{\alpha_{1}}uD^{t}f-\frac{1}{\alpha_{2}}vD^{t'}g=\sum_{j\in J}\gamma_{j}w_{j}\ \mbox{in $k\{X\}$}.
$$
Then it is obvious that $wt(w_{j})=wt(w_{j'}),   \forall  j,   j'\in J$. So whenever we add some non-trivial composition to $S$ while doing the Buchberger-Shirshov algorithm,
any monomial of such a composition will share the same weight. It follows that $S'\subseteq GDN(X)$.

If $w=w_{1}\overline{D^{t_{1}}(u_1D^{m_{1}}s_{1})}=w_{2}\overline{D^{t_{2}}(u_2D^{m_{2}}s_{2})}\in GDN(X)$ is a non-trivial common
multiple of $\overline{D^{t_{1}}(u_1D^{m_{1}}s_{1})}$ and $\overline{D^{t_{2}}(u_2D^{m_{2}}s_{2})}$,
where $s_{1},  s_{2}\in S^{c}$,   $t_1,  t_{2}\in \mathbb{N}$,
$f=u_{1}D^{m_{1}}s_1,   g=u_{2}D^{m_{2}}s_2\in S'$,
then by Theorem \ref{differential cd lemma},   we have
$$
(D^{t_1}f,  D^{t_2}g)_{w}=\frac{1}{\alpha_{1}}w_{1}D^{t_{1}}(u_{1}D^{m_{1}}s_1)-\frac{1}{\alpha_{2}}w_{2}D^{t_{2}}(u_{2}D^{m_{2}}s_2)=\sum_{l\in L}\delta_{l}u_{l}D^{j_{l}}s_{l},
$$
where each  $\delta_{l}\in k,   u_{l}\in [D^{\omega}X],   s_{l}\in S^{c},   j_{l}\in \mathbb{N},    \overline{u_{l}D^{t_{l}}s_{l}}<w$.
Furthermore,   by Lemma \ref{homogeneous},   we can assume that for each $l\in L$,   $wt(\overline{u_{l}D^{j_{l}}s_{l}})=-1$,   which means $(D^{t_1}f,  D^{t_2}g)_{w}\equiv 0 \ mod(S',  w)$.
So $S'$ is a Gr\"{o}bner-Shirshov basis in $GDN(X)$.

It remains to show that the ideal $Id(S) $ of $GDN(X)$ generated by $S$ is $Id(S')$. It
is clear that $S\subseteq Id(S')$. Since $S^{c} \subseteq Id[S^{c}]=Id[S]$,   for any $s\in S^{c}$,   we have $s=\sum \beta_{i}u_{i}D^{t_i}s_{i}$,   where each $\beta_{i}\in k,   u_{i}\in [D^{\omega}X],   s_i\in S$
and $wt(\overline{u_{i}D^{t_i}s_{i}})=wt(\overline{s})$. By Lemma \ref{Novikov ideal},   it follows that $S'\subseteq Id(S) $.

(ii).  Since
$\{w\in [D^{\omega}X] \mid w\neq u\overline{D^{t}s},  \ u\in [D^{\omega}X],   t\in \mathbb{N},   s\in S',  wt(w)=-1\}
=\{w\in [D^{\omega}X]\mid w\neq u\overline{D^{t}s},    \ u\in [D^{\omega}X],   t\in \mathbb{N},   s\in S^{c},  wt(w)=-1\}$ by (i),   we have
$Irr(S')\subseteq Irr[S^{c}]$. The result follows immediately.
\end{proof}

\begin{remark}
Theorem \ref{Novikov PBW} essentially offers another way to calculate Gr\"{o}bner-Shirshov basis in $GDN(X)$ and it indicates some close connection between $GDN(X|S)$ and its universal enveloping algebra $k\{X|S\}$. In fact, by Lemma \ref{Novikov ideal}, we have $GDN(X)\cap Id[S]=Id(S)$. It is clear that $Id[S]$ is a subalgebra of $(k\{X\},\circ)$ as Gelfand-Dorfman-Novikov algebra. Then we have a Gelfand-Dorfman-Novikov algebra isomorphism as follows:
$$
GDN(X)/ Id(S)=GDN(X)/(Id[S]\cap GDN(X))\cong(GDN(X)+Id[S])/Id[S]\leq (k\{X|S\},\circ).
$$

\end{remark}
\subsection{Algorithms for word problems}
 The general observation shows that for a homogeneous variety the word problem in an algebra with finite number of homogeneous relations is always algorithmically solvable. In this subsection,   we will provide   algorithms for solving such word problems.

Let $k\{X|S\}$  be a commutative differential algebra and $S=\{f_{i}\mid 1\leq i\leq p\},   p\in \mathbb{N}$,   where  $S$ is $D\cup X$-homogeneous
in the sense that for any $f=\sum^{q}_{j=1}\beta_jw_{j}\in S$,   we have $|w_{1}|_{D\cup X}=|w_{2}|_{D\cup X}=\dots =|w_{q}|_{D\cup X}$.

In this subsection,   we always assume that $S\subset k\{X\}$ is a non-empty $D\cup X$-homogeneous set.
 We call $S$ a minimal set,   if there are no $f,   g\in S$ with $f\neq g$,   such that $\overline{f}=u\overline{D^{t}g}$ for any $u\in [D^{\omega}X],   t\in \mathbb{N}$.
For any $f, g\in S$,   if $\overline{f}=u\overline{D^{t}g}$ and the composition $[f,  D^{t}g]_{\overline{f}}=\frac{1}{\alpha_{1}}f-\frac{1}{\alpha_{2}}uD^{t}g\equiv 0 \ mod(S,  w)$,
then we delete $f$ from $S$ to reduce the set $S$ in one step to a new set $S_0$,   i.e.,   $S\longrightarrow S_{0}=S\setminus\{f\}$; If $\overline{f}=u\overline{D^{t}g}$ and the
 composition $[f,  D^{t}g]_{\overline{f}}=\frac{1}{\alpha_{1}}f-\frac{1}{\alpha_{2}}uD^{t}g \not\equiv 0 \ mod(S,  w)$,
 then we replace $f$ by $h\triangleq \frac{1}{\alpha_{1}}f-\frac{1}{\alpha_{2}}uD^{t}g$ to reduce the set $S$ in one step to a new set $S_0$,
 i.e.,   $S\longrightarrow S_{0}=(S\setminus \{f\})\cup \{h\}$,   where $ \alpha_{1}=LC(f)$ and $\alpha_{2}=LC(D^{t}g)$.
In both cases,    we  say that $f$ is reduced by $g$. It is clear that $S_0$ is also a $D\cup X$-homogeneous set.

\begin{lemma}\label{keep trivial}
If $|S|<\infty$ and $S$ is $D\cup X$-homogeneous,   then we can effectively reduce $S$ into a minimal $D\cup X$-homogeneous set $S^{(0)}$ in finitely many steps,
 such that $Id[S]=Id[S^{(0)}]$ and for any $f\in S$,   we have $f=\sum\beta_{q}u_{q}D^{t_{q}}s_{q}$,   with $|\overline{u_{q}D^{t_{q}}s_{q}}|_{D\cup X}=|\overline{f}|_{D\cup X}$ and $\overline{u_{q}D^{t_{q}}s_{q}}\leq \overline{f}$,
 where each $\beta_{q}\in k,  u_q\in [D^{\omega}X],   s_{q}\in S^{(0)},   t_{q}\in \mathbb{N}.$
 \end{lemma}
\begin{proof}
Suppose $S=\{f_{i}\mid 1\leq i\leq p\},   p\in \mathbb{N}$ and $1\leq |\overline{f_{1}}|_{D\cup X}\leq|\overline{f_{2}}|_{D\cup X}\leq\dots \leq|\overline{f_{p}}|_{D\cup X}.$
Given $f,  g\in S$,   suppose
$$\overline{f}=a_{n}[i_n]\cdots   a_{1}[i_1] \ \mbox{and } \overline{g}=b_{m}[j_m]\cdots   b_{1}[j_1],  $$
with
 $a_{n}[i_n]\geq \dots \geq a_{1}[i_1],  b_{m}[j_m]\geq \dots \geq b_{1}[j_1] \ \mbox{and }j_m\leq i_n.$
To decide whether $g$ can reduce $f$ or not,   we only need to check whether one of $\overline{g}$,   $\overline{D^{1}g}$,   \dots  ,   $\overline{D^{i_n-j_m}g}$ is a subword of $\overline{f}$ or not.
Define $ord(S)=(p,  \overline{f_{p}},   \overline{f_{p-1}},   \dots  ,   \overline{f_{1}} )$. Then if one reduce $S$ in one step to $S_{01}$,   we have $ord(S_{01})<ord(S)$ lexicographically.
 Therefore,   $S$ can be reduced into a minimal set $S^{(0)}$ in finitely many steps,   say,   $S\longrightarrow S_{01}\longrightarrow S_{02} \longrightarrow\dots \longrightarrow S_{0l}=S^{(0)}$.
 Then by induction on $l$,   we easily get each $Id[S_{0m}]=Id[S]$ and for any $f\in S$,   we have $f=\sum\beta_{q}u_{q}D^{t_{q}}s_{q}$,   with $|\overline{u_{q}D^{t_{q}}s_{q}}|_{D\cup X}=|\overline{f}|_{D\cup X}$
 and $\overline{u_{q}D^{t_{q}}s_{q}}\leq \overline{f}$,   where $ 1\leq m\leq l,  \beta_{q}\in k,   u_q\in [D^{\omega}X],  s_{q}\in S_{0m},   t_{q}\in \mathbb{N}.$
\end{proof}

Suppose that $S$ is a minimal set and
$$S=S^{(0)}=\{f_{i}\mid 1\leq i\leq p\},   p\in \mathbb{N},  $$
where
 $1\leq |\overline{f_{1}}|_{D\cup X}\leq|\overline{f_{2}}|_{D\cup X}\leq\dots \leq|\overline{f_{p}}|_{D\cup X}.$
For any $f,  g\in S^{(0)},    t_1,  t_2 \in \mathbb{N},   t_1,  t_2 \leq 1,   w=lcm(\overline{D^{t_{1}}f},  \overline{D^{t_{2}}g})$,
we will check composition $[D^{t_1}f,  D^{t_2}g]_{w}$ whenever $w$ is a non-trivial common multiple of $\overline{D^{t_{1}}f} \mbox{ and } \overline{D^{t_{2}}g}$.
If all such compositions are trivial,   we just set $S_1= S^{(0)}$. Otherwise,   if for some $t_1,  t_2 \leq 1,   w=lcm(\overline{D^{t_{1}}f},  \overline{D^{t_{2}}g})$,   the non-trivial composition
$
 [D^{t_1}f,  D^{t_2}g]_{w}=\sum_{i\in I}\beta_i w_i=h,
 $
then $|\overline{h}|_{D \cup X}=|w|_{D\cup X} \geq 2$. We collect all such $h$ to make a new set $H_0$ and  denote $S_1=S^{(0)}\cup H_{0}$.
It is clear that each $h\in H_{0}$ is $D\cup X$-homogeneous and we call $|\overline{h}|_{D\cup X}$ the $D\cup X$-length of $h$.
Now we reduce $S_{1}$ to a minimal set $S^{(1)}$.
Noting that $S^{(0)}$ is a minimal set,   if some inclusion composition is not trivial,   then it must involve
some element that is not in $S^{(0)}$. Furthermore,   each $h\in H_0$  has $D\cup X$-length at least $2$,   so every non-trivial inclusion composition that is
added also has $D\cup X$-length at least $2$. So if we
denote $S^{(1)}=S^{(0)}_{sub}\cup R^{(0)}$,   where $S^{(0)}_{sub}=S^{(1)}\cap S^{(0)}$ and $R^{(0)}=S^{(1)}\setminus S^{(0)}$,   then we get
 each $r\in R^{(0)}$,   $|\overline{r}|_{D \cup X}\geq 2$. For any $f,  g\in S^{(0)}_{sub}$,   if
$$
[D^{t_1}f,  D^{t_2}g]_{w}\equiv 0\ mod(S_1,   w),
$$
then
 $$
 [D^{t_1}f,  D^{t_2}g]_{w}\equiv 0\ mod(S^{(1)},   w)
 $$
by Lemma \ref{keep trivial}.
Continue this progress,   and suppose
 $$S_{n}=S^{(n-1)}\cup H_{n-1},  \  S^{(n)}=S^{(n-1)}_{sub}\cup R^{(n-1)},  $$
 where  $S^{(n)}$ is a minimal set and for any $h\in H_{n-1}$,   $  r\in R^{(n-1)}$,    $ |\overline{h}|_{D \cup X}\geq n+1,  \ |\overline{r}|_{D \cup X}\geq n+1.$
Then in order to get $S_{n+1}$,   for any $f,  g\in S^{(n)},  t_1,  t_2\in \mathbb{N},   t_1,  t_2 \leq n+1,   w=lcm(\overline{D^{t_{1}}f},  \overline{D^{t_{2}}g})$,
 we need to check composition $[D^{t_1}f,  D^{t_2}g]_{w}$ whenever $w$ is non-trivial.
If all such compositions are trivial,   we just set $S_{n+1}= S^{(n)}$; Otherwise,   say $[D^{t_1}f,  D^{t_2}g]_{w}=h$ is not trivial.
If $f,  g\in S^{(n-1)}_{sub}\subseteq S_{n},   \ 0 \leq  t_1,  t_2 \leq n $,
then by the construction,   $[D^{t_1}f,  D^{t_2}g]_{w}\equiv 0 \ mod(S_{n},  w )$,   and by Lemma \ref{keep trivial},
we get $[D^{t_1}f,  D^{t_2}g]_{w}\equiv 0 \ mod(S^{(n)},  w )$. Therefore,   if we have a non-trivial composition,   at least one of $f$ and $g$ is in $R^{(n-1)}$,   or at least one of $t_{1}$ and $t_{2}$ equals $n+1$.
Thus,   if we denote $S_{n+1}=S^{(n)}\cup H_{n}$,   then for any $h\in H_n$,   $|\overline{h}|_{D \cup X}\geq n+2$.
By the same reasoning as above,   if we continue to reduce $S_{n+1}$ to a minimal set  $S^{(n+1)}=S^{(n)}_{sub}\cup R^{(n)}$,   then we get for all $r\in R^{(n)},  \ |\overline{r}|_{D \cup X}\geq n+2.$
As a result,   we get the following lemma.

\begin{lemma}\label{length k}
For any $f\in S^{(n)}$,   if $|\overline{f}|_{D \cup X}\leq n,  $ then  $f\in S^{(l)}$ for any $l\geq n.$
\end{lemma}
\begin{proof}
Noting that after we get $S^{(n)}$,   any composition that may be added afterwards  has $D\cup X$-length more than $n$,   but $f$ can not be reduced
by any element which has $D\cup X$-length more than $n$ or by element in $ S^{(n)}\setminus \{f\}$.
\end{proof}

Define
 $$ \widetilde{S}=\{ f\in \bigcup_{n\geq 0}S^{(n)}  \mid f\in S^{(|\overline{f} |_{D \cup X})}  \} .$$
Then by Lemma \ref{length k},   we have
  $$ \widetilde{S}=\{ f\in \bigcup_{n\geq 0}S^{(n)}  \mid f\in S^{(l)},   \ for \ any \ l\geq |\overline{f} |_{D \cup X}     \}.$$

\begin{lemma}\label{GSB}
$Id[S]=Id[\widetilde{S}]$ and $\widetilde{S}$ is a Gr\"{o}bner-Shirshov basis in $k\{X\}$.
\end{lemma}
\begin{proof}
Since $Id[S]=Id[S^{(0)}]=\dots =Id[S^{(n)}]$ for any $ n\geq 0$,   we have $Id[\widetilde{S}]\subseteq Id[S]$. On the other hand,
for any $f\in S$,   if $|\overline{f}|_{D\cup X}=n$,   then by Lemma \ref{keep trivial},   $f=\sum\beta_{q}u_{q}D^{t_{q}}s_{q}$,
where each $s_{q}\in S^{(n)}$ and $|\overline{s_{q}}|_{D \cup X}\leq n$,   i.e.,   $s_{q}\in \widetilde{S}$.
Therefore,   $Id[S] =Id[\widetilde{S}]$. For
any $f,  g\in \widetilde{S},   \ t_{1},   t_{2}\in \mathbb{N},  \ w=lcm(\overline{D^{t_{1}}f},  \overline{D^{t_{2}}g})$,   let $l\triangleq |\overline{f}|_{D \cup X}+|\overline{g}|_{D \cup X}+t_1+t_2.$
If there exists composition $[D^{t_1}f,  D^{t_2}g]_{w}$,   then $[D^{t_1}f,  D^{t_2}g]_{w}\equiv 0 \ mod(S_{l+1},  w)$ by construction.
And by Lemma \ref{keep trivial},   we have $[D^{t_1}f,  D^{t_2}g]_{w}\equiv 0 \ mod(S^{(l+1)},   w)$,   i.e.,   $[D^{t_1}f,  D^{t_2}g]_{w}=\sum\beta_{i}u_{i}D^{t_{i}}s_{i}$,
 where each $s_{i}\in S^{(l+1)}$ and $|\overline{s_{i}}|_{D \cup X}\leq |w|_{D\cup X}< l+1.$
Thus by the definition of $\widetilde{S}$,    we get $[D^{t_1}f,  D^{t_2}g]_{w}\equiv 0 \ mod(\widetilde{S},   w)$.
\end{proof}

\begin{proposition}\label{com diff solvable}
If $|S|<\infty$ and $S$ is $D\cup X$-homogeneous,   then $k\{X|S\}$ has a solvable word problem.
\end{proposition}
\begin{proof}
For any $f=\sum\beta_i w_i \in k\{X\}$,   where $w_1>w_2>\dots  $. We may assume that $|w_{i}|_{D \cup X}\leq n $.   By Lemma \ref{GSB} and Theorem \ref{differential cd lemma},
 $f\in Id[\widetilde{S}]$ implies that $w_{1}=u\overline{D^{t}s}$ for some $s\in \widetilde{S}$.
Moreover,   if $w_{1}=u\overline{D^{t}s}$ for some $s\in \widetilde{S}$,   then $s\in S^{(n)}$. Note that $S^{(n)}$ is a finite $D\cup X$-homogeneous set that can be constructed effectively from $S$.
After reducing $f$ by such $s$,   we get a new polynomial
$$
f'=f-\frac{\beta_{1}}{LC(uD^{t}s)}uD^{t}s=\sum\beta'_{i'} w_{i'},
$$
with each $|w_{i'}|_{D \cup X}\leq n $. Continue to reduce $f'$ by elements in $S^{(n)}$.
If finally we reduce $f'$ by $S^{(n)}$ to $0$,   then $f\in Id[S]$. Otherwise,   $f\notin Id[S]$.
In particular,   if $w_{1}\neq u\overline{D^{t}s}$ for any $s\in S^{(n)},t\in \mathbb{N}$,   then $w_{1}\neq u\overline{D^{t}s}$ for any $s\in \widetilde{S},t\in \mathbb{N}$, and thus $f\notin Id[S]$.
\end{proof}

\ \

Since for any $f\in GDN(X)$,   if
$f=\sum_{1\leq i \leq n}\beta_i w_i $ is homogeneous in the sense that $|w_{1}|=\dots =|w_{n}|,  $
 then $f$ is $D\cup X$-homogeneous because $|w|_{D\cup X}=2|w|+wt(w)$ for any $w\in [D^{\omega}X]$. Given $GDN(X|S)$,   if $|S|<\infty$ and $S$ is homogeneous,
  then taking $S$ as a subset of $k\{X\}$,   $S$ is $D\cup X$-homogeneous.
Thus we can get a Gr\"{o}bner-Shirshov basis $\widetilde{S}$ in $k\{X\}$.
Then by Theorem \ref{Novikov PBW} and Proposition \ref{com diff solvable},
we immediately get the following proposition.

\begin{proposition}If $|S|<\infty$ and $S\subseteq GDN(X)$ is homogeneous,
then $GDN(X|S)$ has a solvable word problem.
\end{proposition}

\section{A subalgebra of $ GDN(a)$ }

We construct a non-free subalgebra $A$ of the free Gelfand-Dorfman-Novikov algebra $GDN(a)$ over a field of characteristic $0$,   which implies that the variety
of Gelfand-Dorfman-Novikov algebras is not Schreier.

By Proposition 1 in  \cite{variety property},   we immediately get the following lemma.

\begin{lemma}(\cite{variety property}) \label{generator number} The following statements hold:
\begin{enumerate}
\item[(i)] The rank of a free Gelfand-Dorfman-Novikov algebra is uniquely determined,   where the rank means the number of free generators.
\item[(ii)] In a free Gelfand-Dorfman-Novikov algebra of rank $n$,   any set of $n$ generators is a set of free generators.
\item[(iii)] A free Gelfand-Dorfman-Novikov algebra of rank $n$ can't be generated by less than $n$ elements.
\end{enumerate}
\end{lemma}

In this subsection,   we consider the free Gelfand-Dorfman-Novikov algebra $GDN(a)$ generated by one element $a$.
\begin{theorem}
Let $A= \langle a\circ a,   (a\circ a)\circ a,   ((a\circ a)\circ a)\circ a \rangle $ be the subalgebra of the free Gelfand-Dorfman-Novikov algebra $GDN(a)$ generated by the set $\{a\circ a,   (a\circ a)\circ a,   ((a\circ a)\circ a)\circ a\}$. Then $A$ is not free.
\end{theorem}
\begin{proof}
Suppose that $A$ is free. Then by Lemma \ref{generator number} $(iii)$,   we get $rank(A)\leq 3.$

If $rank(A)=3,  $ then by Lemma \ref{generator number},   $a\circ a,   (a\circ a)\circ a,   ((a\circ a)\circ a)\circ a$ are free generators. However,
$$(a\circ a)\circ (((a\circ a)\circ a)\circ a)=((a\circ a)\circ a)\circ ((a\circ a)\circ a),  $$
which means that $a\circ a,   (a\circ a)\circ a,   ((a\circ a)\circ a)\circ a$ are not free generators.

If $rank(A)=1$ and
$$f=\beta_{1}(a\circ a) + \beta_{2} (a\circ a)\circ a  +\sum\beta_{i}w_{i} $$
is a free generator of $A$,   where each $w_{i}$ has length at least $4$,   then
$$a\circ a= \gamma_{1}{f}+\sum\gamma_{j}f\circ f\circ f\circ \cdots \circ  f,  $$
where  $f$ occurs at least twice in each term of the second summand on the right side  and each of them is with some bracketing.
We can rewrite this formula to the following form:
$$a[0]a[-1]= \gamma_{1}f+\sum\lambda_{i_1,  i_2,  \dots  ,   i_{n}}(D^{i_1}f)(D^{i_2}f)\cdots (D^{i_n}f),   $$
where $i_{1}\geq i_2\geq \dots  \geq i_n\geq 0,   n\geq 2$. Then
each term in the second summand has leading term bigger than $a[0]a[-1]$. Since
$$\overline{(D^{i_1}f)(D^{i_2}f)\cdots (D^{i_n}f)}=(\overline{D^{i_1}f})(\overline{D^{i_2}f})\cdots (\overline{D^{i_n}f}),  $$
by analysing the leading terms of the left side and the right side,   we get each $\lambda_{i_1,  i_2,   \dots   ,   i_{n}}=0,  $ so $a[0]a[-1]= \gamma_{1}f,  $
i.e.,   $f=\frac{1}{\gamma_{1}}a\circ a$. However,   $(a\circ a)\circ a\notin \langle a\circ a\rangle = A.$ This is a contradiction.

If $rank(A)=2,  $ suppose
$$ f_{1}=\beta_{1}(a\circ a) + \beta_{2} ((a\circ a)\circ a)  +\sum\beta_{e}w_{e},       $$
$$f_{2}=\gamma_{1}(a\circ a) + \gamma_{2} ((a\circ a)\circ a)  +\sum\gamma_{e'}w_{e'},      $$
are free generators,   where each  $w_{e},  w_{e'}$  has length at least $4$. Say
$$
a\circ a=\lambda_{1}f_{1}+\lambda_{2}f_{2}+\sum{\lambda_{j_1,  j_2,   \dots   ,   j_{n}}}f_{j_1}\circ f_{j_{2}}\circ\cdots \circ f_{j_n},    $$
$$
(a\circ a)\circ a =\mu_{1}f_{1}+\mu_{2}f_{2}+\sum{\mu_{q_1,  q_2,   \dots   ,   q_{m}}}f_{q_1}\circ f_{q_{2}}\circ \cdots \circ f_{q_m},
$$
$$
((a\circ a)\circ a)\circ a =\nu_{1}f_{1}+\nu_{2}f_{2}+\sum{\nu_{l_1,  l_2,  \dots ,   l_{r}}}f_{l_1}\circ f_{l_{2}}\circ\cdots \circ f_{l_r},
$$
where $  j_{1},   \dots    ,   j_{n},   q_{1},   \dots    ,   q_{m},   l_{1},   \dots    ,   l_{r}\in \{1,  2\},  $ and each term in the third summand on the right side of each equation is with some bracketing.
Rewriting the right sides into linear combination of basis of the free Gelfand-Dorfman-Novikov algebra $GDN(a)$ and comparing terms of length $2$ and $3$ on the left sides and the right sides,   we get
 \begin{gather*}\setlength\textwidth{13.50pt}
  \begin{pmatrix}  \beta_1 &  \gamma_1 \\ \beta_2 &  \gamma_2 \end{pmatrix}
\begin{pmatrix}  \lambda_1 &  \mu_1 \\ \lambda_2 &  \mu_2 \end{pmatrix} =
\begin{pmatrix}  1 &  0 \\ 0 &  1 \end{pmatrix}
 \end{gather*}
and
\begin{gather*}\setlength\textwidth{13.50pt}
  \begin{pmatrix}  \beta_1 &  \gamma_1 \\ \beta_2 &  \gamma_2 \end{pmatrix}
\begin{pmatrix}   \nu_1 \\ \nu_2 \end{pmatrix} =
\begin{pmatrix}  0\\  0 \end{pmatrix}.
 \end{gather*}
So $\nu_1 = \nu_2=0$ and
$$
((a\circ a)\circ a)\circ a =\sum{\nu_{l_1,  l_2,  \dots ,    l_{r}}}f_{l_1}\circ f_{l_{2}}\circ\cdots \circ f_{l_r}.
$$
However,   among the terms of the right side,   only $(a\circ a)\circ (a\circ a )$ has length $4$,   but $((a\circ a)\circ a)\circ a\neq \beta(a\circ a)\circ (a\circ a ),  $ for any $\beta \in k.$

Therefore,   $A$ is not free.
\end{proof}
\section*{Acknowledgment} We would like to thank the  referee for  valuable suggestions.

\end{document}